\documentclass{amsart}
\usepackage{graphicx}
\usepackage{tikz}
\usetikzlibrary{shapes,arrows}
\tikzstyle{block} = [rectangle, draw, fill=white!10,
    text width=2cm, text centered, rounded corners, minimum height=1cm]
\tikzstyle{line} = [draw, very thick, color=black!50, -latex']
\newtheorem{thm}{Theorem}[section]
\newtheorem{cor}[thm]{Corollary}
\newtheorem{lem}[thm]{Lemma}

\newtheorem{ex}[thm]{Example}

\theoremstyle{definition}
\newtheorem{defn}[thm]{Definition}
\theoremstyle{remark}
 \newtheorem{rem}[thm]{Remark}
\numberwithin{equation}{section}

\begin{document}
\title[GENERALIZATION OF $RAD$-$D_{11}$-MODULE]
{GENERALIZATION OF $RAD$-$D_{11}$-MODULE}%
\author{MAJID MOHAMMED ABED, ABD GHAFUR AHMAD and A. O. Abdulkareem}%
\address{School of mathematical Sciences, Faculty of Science and
Technology,Universiti Kebangsaan Malaysia,43600 UKM, Selangor
Darul Ehsan, Malaysia.}%
\email{m\_m\_ukm@ymail.com.(Corresponding author)}%
\email{ghafur@ukm.my}%
\email{afeezokareem@gmail.com}
\subjclass[2000] {54C05, 54C08, 54C10\\Corresponding author majid mohammed abed }%
\begin{abstract}
This paper gives generalization of a notion of supplemented module. Here, we utilize some algebraic properties like supplemented, amply supplemented and local modules in order to obtain the generalization. Other properties that are instrumental in this generalization are $D_{i}$, $SSP$ and $SIP$. If a module $M$ is $Rad$-$D_{11}$-module and has $D_{3}$ property, then $M$ is said to be completely-$Rad$-$D_{11}$-module ($C$-$Rad$-$D_{11}$-module). Similarly it is for $M$ with $SSP$ property. We provide some conditions for a supplemented module to be $C$-$Rad$-$D_{11}$-module.
\end{abstract}
\maketitle

\section{Introduction and Preliminaries}
  Throughout this paper all rings are unital and modules are considered to be right modules. A submodule $N$ of $M$ is small in $M$ ($N \ll M$) if for every submodule $L$ of $M$ with $N+L=M$, $L=M$. In [1], we have that any module $M$ is called hollow module if every proper submodule of $M$ is a small in $M$. The direct summand plays vital role in generalization of supplemented module. A submodule $N$ of $M$ is called supplement of $K$ in $M$ if $N$+$K$=$M$ and $N$ is minimal with respect to this property. A module $M$ is called supplemented if any submodule $N$ of $M$ has a supplement in $M$.\\

  In [2], Y. Talebi and A. Mahmoudi studied $C$-$Rad$-$D_{11}$-module through $Rad$-$\bigoplus$-supplemented modules and $D_{3}$ property to get $C$-$Rad$-$D_{11}$-module. Here we use other algebraic properties such as supplemented, amply supplemented and local modules to get generalization of $Rad$-$D_{11}$-module.\\

  Also summand sum property (SSP) and summand intersection property (SIP) are very important in generalization of supplemented module (see Definition 3.1). From [2], if $M$ is a $Rad$-$D_{11}$-module and has $SSP$ property, then $M$ is a $C$-$Rad$-$D_{11}$-module. If $M_{1}$ and $M_{2}$ are direct summands of $M$ with $M=M_{1}+M_{2}$ and $M_{1}\cap M_{2}$ is also direct summand of $M$ and so $M$ has $D_{3}$ property. Note that, the $D_{3}$-module with $Rad$-$D_{11}$-module already gives a $C$-$Rad$-$D_{11}$-module. There is another module called $T_{1}$-module which has relationship with $C$-$Rad$-$D_{11}$-module. If for every submodule $K$ of $M$ such that $M/K$ is isomorphic to a co-closed submodule of $M$ and every homomorphism $\mu:M\rightarrow M/K$ lifts to a homomorphism $\beta$:$M\rightarrow M$ in this case $M$ is called $T_{1}$-module [4]. If $M$ is a local module then it is a $C$-$Rad$-$D_{11}$-module. On the other hand, $M$ is $C$-$Rad$-$D_{11}$-module if it is projective, supplemented and has $D_{3}$ properties. There have been different notions of generalization of supplemented module conducted by many researchers. These generalizations are motivated by different properties of supplemented module. However, in this study we try to give another notion of generalization for supplemented module. It is interesting to note that several properties of supplemented module have been harnessed to give important properties of the generalization considered.\\
  
  In Section $2$, we give some properties of $C$-$Rad$-$D_{11}$-module. We proved that if $M$ is a projective and local module with $D_2$ property then it is a $C$-$Rad$-$D_{11}$-module. Necessary condition for a supplemented module to have a generalization of rank three is also given. In section $3$, we study three properties (injective, SSP and SIP) of supplemented module over $Rad$-$D_{11}$-module. An easy to follow proof of the consequence of each property is provided. Using unique closure (UC) and extending properties of a module, we give necessary and sufficient condition for $Rad$-$D_{11}$-module to be $C$-$Rad$-$D_{11}$-module 


\section{ALGEBRAIC PROPERTIES AND $C$-$RAD$-$D_{11}$ MODULE}

In this section, we utilize some algebraic properties in order to obtain rank three generalization of supplemented module. Let $M$ be an $R$-module. From [5], a module is said to be a $D_{i}$-module) if it satisfies $D_{i}$ ($i$=1, 2, 3) condition. A module $M$ is called $D_{1}$ if for every submodule $A$ of a module $M$, there is a decomposition $M=M_{1}\bigoplus M_{2}$ such that $M_{1}\in A$ and $A\cap M_{2}$ $\ll M_{2}$.\\
\noindent Equivalently, any module $M$ is called lifting if for all $N$ submodule of $M$ there is a decomposition 
$$M=H\bigoplus G\ni H\in N \ \mbox{and} \ N \cap H\leq M$$.\\
\noindent A module $M$ is called $D_{2}$ if $A\leq M$ such that $M$/$K$ is isomorphic to a summand of $M$ and this implies that $A$ is a summand of $M$.\\
For $N$ and $L$ are submodules of $M$, $L$ is a radical supplement ($Rad$-supplement) of $N$ in $M$ if
 $$N+L=M \ \mbox{and} \ N\cap L\ll Rad(L).$$
 
On the other hand from [6], $M$ is called $D_{11}$-module if every submodule of $M$ has supplement which is a direct summand of $M$. Therefore, any module $M$ is called a $Rad$-$D_{11}$-module if every submodule of $M$ has a $Rad$-supplement that is a direct summand of $M$. Also, a module $M$ is called semiperfect if every factor module of $M$ has a projective cover.
\begin{defn}
A module $M$ is called $C$-$Rad$-$D_{11}$-module if every direct summand of $M$ is $Rad$-$D_{11}$-module.
\end{defn}

 In [7], N.O. Ertas, gives the direct sum of additive Abelian groups $A\bigoplus B$, $A$ and $B$ are called direct summands. The map $$\alpha_{1} :A\rightarrow A \bigoplus B$$ defined by the rule $\alpha(a)= a\bigoplus 0$ is called the injection of the first summand and the map 
 $$p_{1}:A\bigoplus B\rightarrow A$$
\noindent defined by $p_{1}(a\bigoplus b)=a$ is called the projection onto the first summand. Similar maps $\alpha_{2}$, $p_{2}$ are defined for the second summand $B$. Equivalently, the direct sum of objects $A_{i}$ with $\alpha\in I$ is denoted by $A=\bigoplus A_{\alpha}$ and each $A_{\alpha}$ is called direct summand of $A$.
\begin{rem}
Let $A$ and $B$ be two direct summands of an Abelian group $S$ such that $A + B = S$. Then the intersection of $A$ and $B$ is not a direct summand of $S$. An example is given by
$$S =K_{2}\bigoplus K_{8}, \ A=\langle(1,1)\rangle \ \mbox{and} \ B=\langle(0,1)\rangle.$$
Then $S$ is the direct sum of $A$ and $\langle$(1,4)$\rangle$ and of $B$ and $\langle$(1,0)$\rangle$. The intersection is
$$\langle(0,2)\rangle\cong K_{4},$$
which is not a direct summand of S because $S$ is not isomorphic to $K_{4}\bigoplus K_{4}$ or to $K_{2}\bigoplus K_{2}\bigoplus K_{4}$.
The group generated by $A$ and $B$ is $S$ because it contains $(0,1)$ and $(1,0)=(1,1)-(0,1)$.
\end{rem}
From the definitions of $D_{2}$ and $D_{3}$ properties, it is obvious that $D_{2}$ implies $D_{3}$. Other properties like supplemented, amply supplemented and local modules are inherited by summands property. Also Rad-$D_{11}$-module take the same inherited property. It is a fact that there is an equivalence between supplemented and amply supplemented module. Finally, we say that:\\

 (*) If $M$ is a $Rad$-$D_{11}$-module with $D_{2}$ property then it is a $C$-$Rad$-$D_{11}$-module.

\begin{lem}
Let $M$ be an $R$-module. If $M$ has largest submodule then it is a supplemented module.
\end{lem}
The next theorem whose proof shall be given at the end of this section is one of the main results of this study. 
\begin{thm}\label{THM2.4}
Let $M$ be an $R$-module, if $M$ satisfies the following conditions;\\
1. $M$ is a projective module,\\
2. $M$ is local module,\\
3. $M$ is $D_{2}$-module,\\
then $M$ is a $C$-$Rad$-$D_{11}$-module.
\end{thm}
\begin{lem}
Let $M$ be an $R$-module. If $M$ is a projective and supplemented then it is a $Rad$-$D_{11}$-module.
\end{lem}
\begin{proof}
Since $M$ is a projective module with $M=N+K$ then $M$ is a $\pi$-projective. But
$M$ is a supplemented module. Therefore $M$ is amply supplemented.
Since a projective module with amply supplemented property is a semiperfect module, we infer that $M$ is a $Rad$-$D_{11}$-module.
\end{proof}
\begin{thm}
 Let $M$ be a projective and local module. If $M$ is $D_{2}$-module then it is a $C$-$Rad$-$D_{11}$-module.
 \end{thm}

 \begin{proof}
 Let $M$ be a projective and local module. Then from Lemma 2.5 we get $M$ is a $Rad$-$D_{11}$-module. Now we must show that $B$ is a $Rad$-supplemented. In other word, $B$ has a $Rad$-supplement in $A$ that is a direct summand of $A$. Let $A$ be a direct summand of $M$ and $B$, submodule of $A$. Since $M$ is a $Rad$-$D_{11}$-module, then there exists a direct summand $C$ of $M\ni M=B+C$ and $B\cap C\leq Rad(C)$.\\

 \noindent So, $$A=B+(A\cap C).$$
 \noindent But from definition of $D_{2}$-module, if $A\leq M$ such that $M=A$ is isomorphic to a summand of $M$, then $A$ is a
summand of $M$. Thus, $M$ has $D_{3}$ property and therefore $A\cap C$ is a direct summand of $M$. Hence $M_{1}\cap M_{2}$ is also direct summand of $M$.\\
\noindent Thus,
$$B\cap (A\cap C)=B\cap C, \ B\cap C\leq Rad(M) \ \mbox{and} \ B\cap C\leq A\cap C.$$
\noindent So $$B\cap C\leq (A\cap C\cap Rad(M))=Rad(A\cap C).$$
\noindent Consequently, by Definition 2.1, $M$ is a $C$-$Rad$-$D_{11}$-module.
\end{proof}
\begin{thm}
([8]) Let $R$ be a ring. Then the following statements
are equivalent.\\
1. $R$ is a left perfect.\\
2. Every $R$-module $M$ is a supplemented.\\
3. Every projective $R$-module is amply supplemented
\end{thm}
The following theorem gives necessary condition for a supplemented module to possess a ran three generalization.
\begin{thm}
 Let $M$ be projective supplemented module with $D_{2}$ property. If every supplement submodule of
$M$ is a direct summand, then $M$ has generalization of rank three.
\end{thm}

\begin{proof} Since every supplement submodule of $M$ is a direct summand, $M$ is a
$D_{11}$-module. But $M$ is a supplemented module then it is a strongly-$D_{11}$-module and
thus, R is a perfect ring. From Theorem 2.7 we get $M$ is amply supplemented. But $M$ is projective module. Hence, $M$ is $Rad$-$D_{11}$-module. As a consequent of (*), $M$ has a generalization of rank three ($C$-$Rad$-$D_{11}$-module).
\end{proof}
\begin{thm}
Let $M$ be a $Rad$-$D_{11}$-module with $D_{1}$ property. If $M=M_{1}\bigoplus M_{2}$ is a direct sum of submodules $M_{1}$ and $M_{2}$, then $M_{1}$ and $M_{2}$ are relatively projective and so $M$ is a $C$-$Rad$-$D_{11}$-module.
\end{thm}
\begin{proof}
By [9] and Lemma 2.5.
\end{proof}

\begin{thm}
Let $M$ be an $R$-module. If $M$ satisfies the following conditions:\\
1. $M$ is a projective module;\\
2. $M$ is a semiperfect module;\\
3. $M$ is $D_{3}$-module;\\
then $M$ is a $C$-$Rad$-$D_{11}$-module.
\end{thm}

\begin{proof}
Let $A\leq M$. Then by assumption, there exists a projective cover $\varphi:P\rightarrow M/K$ and there is an epimorphism $\varphi:M \rightarrow M/K$.\\
\noindent Since $M$ is projective then there exists a homomorphism $$\mu:M \rightarrow P\ni\varphi\circ\mu=\beta.$$
\noindent Also since $\varphi$ is small and is an epimorphism, $\mu$ splits ($P$ is projective). We have a homomorphism $$g:P\rightarrow M\ni \mu\circ g=1_{P} \ \mbox{and} \ \varphi=\varphi\circ\mu\circ g=\delta\circ g.$$
\noindent Since\\
$$M=Ker(\mu)\bigoplus g(P) \ \mbox{and} \ Ker(\mu)\leq A; \mbox{then} M=A+g(p).$$
\noindent Let $\alpha$ be the restriction of  $\varphi$ to $g$($p$). Then $\varphi$=$\alpha$ and so $\alpha$ is an epimorphism. Also since $\varphi$ is small, $\alpha$ is also small. That is
$$Ker(\alpha)=A\cap g(p)\ll g(p)$$
\noindent then $g$($p$) is a supplement of $A$. Thus $M$ is $Rad$-$D_{11}$-module. But $M$ has $D_{3}$ property. Then if $M_{1}$ and $M_{2}$ are direct summands of $M$ with
$$M=M_{1}+M_{2} \ \mbox{and} M_{1}\cap M_{2}.$$
\noindent is also direct summand of $M$. Thus $M$ is a $C$-$Rad$-$D_{11}$-module.
\end{proof}
The following is an example of matrix over $Rad$-$D_{11}$-module which gives $C$-$Rad$-$D_{11}$-module.

\begin{ex}
 Let $M_{4\times4}$ be a matrix over field $F$ such that it satisfies $D_{2}$ property. 
\[
    M=
      \begin{bmatrix}
        a & 0 & 0 & 0 \\
        y & b & 0 & 0 \\
        0 & 0 & b & 0 \\
        0 & 0 & x & 0
      \end{bmatrix}
    \]
\end{ex}
such that a, b, x, y in $F$.

\begin{cor}
 Let $A$ and $B$ submodules of projective $R$-module $M$. If $B$ is a minimal
with respect to the property $A+B=M$ then $M$ is $C$-$Rad$-$D_{11}$-module.
\end{cor}
\begin{proof}
Let $A$ and $B$ be submodules of $M$ and let $A$ be supplement of $B$ in $M$. Then
$M$ is a supplemented module. We need to show that the homomorphism $\beta$ from $M$ into
$M$ is a split homomorphism. Let $1_{M}$ be identity mapping. Let $M$ be projective module then there exists
a mapping from $M$ into $A\bigoplus B$ such that $\beta\circ\gamma=1_{M}$. This means $\beta$ is split ($M$
is a $\pi$-projective). Let $g$ belongs to endomorphism of $M$ such that $Im(g)\subset A$
and $Im(1-g) \subset B$ with $M=A+B$. Since $g(A)\subseteq A$, $M=A+g(B)$ and
$g(A\cap B)=A\cap g(B)$ such that $n=g(b)$ then $b-n=(1-g)(b)$ and $b\in A$. Since
$A\cap B\ll A$ and $A\cap g(A) \ll g(A)$ where $g(A)$ is a supplement of $A$ and $g(A)\subset
B$, $M$ is $C$-$Rad$-$D_{11}$-module.
\end{proof}
\begin{cor}
Let $M$ be a $Rad$-$D_{11}$-module. If $M$ is a $T_{1}$-module then it is a $C$-$Rad$-$D_{11}$-module.
\end{cor}

\begin{rem}
We are now ready to prove the main result of this section, Theorem \ref{THM2.4}. This theorem further emphasizes the role of $D_{i}$-modules in our rank three generalization of supplemented module.
\end{rem}

\begin{proof}[Proof of Theorem \ref{THM2.4}]
Let $M$ be a projective module with $M=N+K$. Then $M$ is $\pi$-projective. Since $M$ has largest submodule then $M$ is a local module. $M$ contains all proper submodules such that $A\subseteq Rad(M)\ll M$ and
so $A\ll M$. Hence $M$ is a hollow module ($A$ submodule of
$M$ then $A+M=M$). Again by definition of hollow module we get $A\cap M=A$. We
have $A\ll M$ therefore $A$ supplement in $M$ and so $M$ is a
supplemented module. Hence $M$ is amply supplemented, with projective property, implies that $M$ is semiperfect and so is a $Rad$-$D_{11}$-module. Since $M$ has $D_{2}$ property, then it is $D_{3}$-module. Thus by Lemma 2.5 $M$ is a $C$-$Rad$-$D_{11}$-module.
\end{proof}
\section{INJECTIVE, $SSP$ AND $SIP$ PROPERTIES OVER $RAD$-$D_{11}$-MODULE}

In this section,our attention is drawn to three properties of supplemented modules; injective, $SSP$ and $SIP$. Here we investigate these properties over $Rad$-$D_{11}$-module for the purpose of our notion of generalization of supplemented module.
\begin{defn}
A module $M$ is said to have the summand sum property $SSP$ if the sum of any pair of direct summands of $M$ is a direct summands of $M$, i.e., if $N$ and $K$ are direct summands of $M$ then $N$+$K$ is also a direct summand of $M$.
\end{defn}
\begin{lem}
  ([2]) Let $M$ be a $Rad$-$D_{11}$-module. If $M$ has $SSP$ property then $M$ is a $C$-$Rad$-$D_{11}$-module.
\end{lem}
Let $N$ be a submodule of left $R$-module $M$. Hence there exists submodule $L$ of $M$ where $M$ is the internal direct sum of $N$ and $L$. In other words,
 $N+L=M$ and $N\cap L={0}$. This implies that $M$ is an injective module. Also, a module $P$ is called projective if and only if for every surjective module homomorphism $f:M\rightarrow P$ there exists a module homomorphism $h:P\rightarrow M$ such that $fh$=$id_{P}$.
\begin{thm}
Let $M$ be an $R$-module. If $A$ and $B$ any two direct summands of $M$ such that $A\cap B$ is injective $R$-module then $M$ is a $C$-$Rad$-$D_{11}$-module.
\end{thm}
\begin{proof}
Let $A$ and $B$ be direct summands of $M$. By hypothesis $M$ is injective module because $M\cap M=M$. Therefore any direct summand of $M$ is injective and so $A$ and $B$ are also injective. Again by hypothesis $M=\bigoplus K$ for some $K\leq M$. Hence, $$A=A\cap B\bigoplus A\cap K.$$

\noindent Also,
$$B=(A\cap B)\bigoplus(B\cap K).$$

 \noindent Thus,
$A\cap B$, $A\cap K$ and $B\cap K$ are injective. \\

 \noindent We have
$$A+B=(A\cap B)\bigoplus(A\cap K\bigoplus B\cap K).$$
 \noindent Then, it follows that $A+B$ is injective and so it is a direct summand of $M$, $M$ has $SSP$ property. Thus, from Lemma 3.2, $M$ is a $C$-$Rad$-$D_{11}$-module.
\end{proof}

\begin{cor}
Let $M$ be a projective and supplemented $R$-module. If $M=A\bigoplus B$ $(direct \ summand \ of \ M)$ and $(A\cap B)$ is an injective $R$-module,
 then $M$ is a $C$-$Rad$-$D_{11}$-module.
\end{cor}
\begin{lem}
 ([9]) Let $R$ be a ring. If $R$ is a semisimple then every $R$-module $M$ has $SSP$ property.
\end{lem}
\begin{defn}
Any module $M$ has $C_{3}$ property if $M_{1}$ and $M_{2}$ are summands of $M$ such that $M_{1}\cap M_{2}=0$ then $M_{1}\bigoplus M_{2}$ is a summand of $M$.
\end{defn}
Recall that an $R$-module $M$ has the summand intersection property $SIP$ if the intersection of two summands is again a summand. Let $M$ be a projective $R$-module then $M$ has the $SIP$ property if and only if for any direct summands $A$ and $B$ of $M$, $A+B$ is a projective $R$-module.
\begin{thm}
Let $M$ be $C_{3}$-module. If $M$ has the $SIP$ then $M$ is a $C$-$Rad$-$D_{11}$-module.
\end{thm}
\begin{proof}
Let $M$ be $C_{3}$-module and has the ($SIP$) property. Let $A$ and $B$ be a direct summands of $M$. We must show that $A+B$ is direct summand of $M$. Since $M$ has the $SIP$ then there exists 
$$D\leq M\ni A\cap B\bigoplus D=M.$$
\noindent By modularity law, we obtain 
$$A=A\cap B\bigoplus D\cap A \ \mbox{and}$$
$$B=A\cap B\bigoplus D\cap A.$$
\noindent Then we have
$$A+B=A\cap B+[D\cap A\bigoplus D\cap A].$$
\noindent Next we prove that
$$(A\cap B)\cap[D\cap A\bigoplus D\cap B]=0.$$
\noindent For if
$$x\in(A\cap B)\cap[D\cap A\bigoplus D\cap A],$$ then
$$x=n_{1}+n_{2} \ \mbox{where} \ n_{1}\in(D\cap A) \ \mbox{and} \ n_{2}\in D\cap B.$$
\noindent We have
$$n_{2}=x-n_{1}\in[A\cap B+D \cap A]\cap D\cap B\bigoplus A\cap D\cap B=0.$$
\noindent Hence,
$$n_{2}=0 \ \mbox{and} \ x=n_{1}.$$ 
\noindent  Now
$$x=n_{1}\in A\cap B\cap D\cap A=A\cap B\cap D=0.$$
\noindent Thus,
$$A+B=A\cap B\bigoplus D\cap A\bigoplus D\cap B=B\bigoplus D\cap A.$$
\noindent Since $M$ has the $SIP$ property and $D$ and $A$ are direct summands, $D\cap A$ is a direct summand. From $C_{3}$ property it follows that $(A+B)=B\bigoplus D\cap A$ is a direct summand of $M$. Thus $M$ has $SSP$ property and from Lemma 3.2 $M$ is a $C$-$Rad$-$D_{11}$-module.
\end{proof}

\begin{cor}
Any projective module $M$ with $C_3$ property over right hereditary ring $R$ is $C$-$Rad$-$D_{11}$-module.
\end{cor}
\begin{proof}
Suppose that $R$ is right hereditary and $M$ is any projective
$R$-module. Since every submodule of a projective $R$-module over right hereditary is projective.
Hence $M$ has the $SIP$. Thus from Theorem 3.7 $M$ is $C$-$Rad$-$D_{11}$-module.
\end{proof}
\begin{thm}
Let $R$ be left hereditary ring. If $M$ is an injective and $Rad$-$D_{11}$-module then it is a $C$-$Rad$-$D_{11}$-supplemented module.
\end{thm}
\begin{proof}
Let $R$ be a left hereditary ring. We must prove that $M$ has $SIP$ property. Factor module of every injective $R$-module is injective. Let $M$ be an injective module which has a decomposition $M=L\bigoplus N$. Let $f$ be a homomorphism from $L$ to $N$. Then $L$ is injective. \\
By assumption,
$$Im(h)\approx(L/Ker(h)) \ \mbox{is} \ \mbox{injective}.$$
\noindent Hence $Im(h)$ is direct summand of $N$. From [10] $M$ has the $SIP$ and so has $SSP$ property. But $M$ is $Rad$-$D_{11}$-module, thus, by Lemma 2.4 $M$ is a $C$-$Rad$-$D_{11}$-module.
\end{proof}
\begin{cor}
Let $M$ be $Rad$-$D_{11}$-module. If every injective $R$-module has the $SIP$ property then $M$ is a $C$-$Rad$-$D_{11}$-module.
\end{cor}
Let $\Lambda$($M$)=$\{$$k$$\in$$K$:$kM$$\subseteq$$M$$\}$. Note that $\Lambda$($M$) is a subring of $K$ containing $R$. For example, if $M$ is the $R$-module $R$, then $\Lambda$($M$)=$R$. On the other hand, if $S$ is any subring of $K$ containing $R$ and $M$ is the $R$-module $S$ then $\Lambda$($M$)=$S$. In particular, $\Lambda$($R_{K}$)=$K$. For integral domain $R$, an $R$-module $M$ is called torsion free if $Ann$($a$)=0, for each 0$\neq$a$\in$$M$ and an $R$-module $M$ is called uniform if every non-zero submodule of $M$ is essential in $M$. According to [6], every finitely generated torsion-free uniform $R$-module is a $C$-$Rad$-$D_{11}$-module. Recall that a submodule $A$ of $M$ is called a fully invariant submodule if $g$($A$)$\subseteq$$A$, for every $g$$\in$$Hom$($M$, $M$)[10]. Moreover; in [11], a module $M$ is called a duo-module if every submodule of $M$ is fully invariant.

\begin{thm}
Let $M$ be $Rad$-$D_{11}$-module. If a commutative domain $R$ is an integrally closed then every finitely generated torsion-free uniform $R$-module is a $C$-$Rad$-$D_{11}$-module.
\end{thm}

\begin{proof}
Suppose that $R$ is integrally closed. Let $T$ be any finitely generated torsion-free uniform $R$-module. Let $k$ be an element in $\Lambda$($T$). Since $kT$$\subseteq$$T$ and $k$ is integral over $R$, then $k$$\in$$R$ and so, $\Lambda$($T$)=$R$. By [11], $T$ is a duo module with $Rad$-$D_{11}$-module lead to $M$ is a $C$-$Rad$-$D_{11}$-module.
\end{proof}

\begin{lem}
 ([2, Lemma 3.4]). Let $M$ be a duo module. Then $M$ has the $SIP$ property.
\end{lem}

In [2], Y. Talebi and A. Mahmoudi calls a module $M$ a $UC$-module if every submodule of $M$ has a unique closure in $M$. A module $M$ is called extending if every closed submodule of $M$ is a direct summand of $M$. Therefore any $UC$-extending module has $D_{3}$ property.
\begin{thm}
Let $M$ be $UC$-extending module. Then $M$ is a $Rad$-$D_{11}$-module if and only if $M$ is a $C$-$Rad$-$D_{11}$-module.
\end{thm}

\begin{proof}
Sufficiency is clear. Conversely, assume that $M$ is $M$-supplemented module. From [1], $M$ has $D_{3}$ property. Hence $M$ is a completely-$Rad$-$D_{11}$-module.
\end{proof}

\begin{lem}
  ([12]) Let $M$ be an $R$-module with $Rad$($M$)=0. If $M$ is a closed weak supplemented module then $M$ is extending.
\end{lem}

\begin{thm}
 For any ring $R$ the following are equivalent:\\

1.  Every left $R$-module is a lifting.\\

2.  Every left $R$-module is extending.
\end{thm}

\begin{thm}
Let $M$ be an $Rad$-$D_{11}$-module with the following conditions:\\

1. $M$ is $UC$-module;\\

2. $Rad$($M$)=0;\\

3. Every nonsingular right $D_{11}$-module is projective;\\

 then $M$ is a $C$-$Rad$-$D_{11}$-module.
\end{thm}
\begin{proof}
Let $M$ be a nonsingular module and $N$ a closed submodule of $M$. Then ($M$/$K$) is nonsingular. Since $M$ is a projective then $N$ is a direct summand of $M$. From [12], $M$ is closed weak supplemented with $Rad$($M$)=0 implies that $M$ is extending module. Now from condition (1) with $Rad$-$D_{11}$-module we obtain $M$ is a $C$-$Rad$-$D_{11}$-module. (see Theorem 3.13).
\end{proof}

\begin{thm}
Let $M$ be an $R$-module. If $M$ satisfies the following conditions:\\

1.  $M$ is $UC$-module;\\

2. $M$ is $Rad$-$D_{11}$-module;\\

3. $M$ is lifting module;\\

then $M$ is a $C$-$Rad$-$D_{11}$-module.
\end{thm}

\begin{cor}
Let $M$ be an $R$-module. If $M$ is a local then $M$ is a $C$-$Rad$-$D_{11}$-module.
\end{cor}
\begin{cor}
Let $M$ be an $R$-module such that every direct summand of $M$ is a finite direct sum of hollow modules. If $M$ has $D_{3}$ property then $M$ is a $C$-$Rad$-$D_{11}$-module.
\end{cor}
\begin{cor}
Any duo module $M$ has $C3$ property is $C$-$Rad$-$D_{11}$-module
\end{cor}

\begin{rem}

We are now ready to prove the main result of this section, Theorem 3.7. The classification of $Rad$-$D_{11}$-module is very important in the process of generalization of supplemented module:
\end{rem}

Let $S$=$\{$$s_{1}$, \dots,$s_{n}$$\}$$\subset$ $M$ be a set of generators for $M$ over $End_{R}$($M$). Since a direct sum of semisimple modules is also a sum of simple modules then it is semisimple. So every direct sum of semisimple modules is again semisimple. Hence $M^{S}$ is a semisimple module. Moreover; we have $R$-homomorphism $R$ $\rightarrow$ $M^{S}$ and $r$ $\mapsto$($rs_{1}$, \dots,$rs_{n}$) is injective: Suppose that $rs_{i}$=0 for all generators of  $R$ as an $End_{R}$ $M$-module. \\

\noindent Therefore we can write every\\

\hspace{22 mm} $x\in M$ as $\beta_{1}$($s_{1}$)+ \dots+$\beta_{n}$($s_{n}$) for $\beta_{i}\in End_{R}M$. \\

\noindent So we have,

  $$rx=r(\beta_{1}(s_{1})+ \cdots +\beta_{n}(s_{n}))=\beta_{1}(rs_{1})+ \cdots+\beta_{n}(rs_{n})=0.$$

  \noindent Also we see that $rx=0 \ \forall \ x \in M$. By the faithfulness of $M$, we conclude that $r=0$. This shows that $R$ is (as an $R$-module) isomorphic to a submodule of $M^{S}$. Hence $R$ is a semisimple and by Lemma 3.5 we obtain $M$ has $SSP$ property. But $M$ is $Rad$-$D_{11}$-module. Thus $M$ is a $C$-$Rad$-$D_{11}$-module. Theorem 3.7 is thus proved.

\section{Conclusion}
The supplemented module is very important in module theory specially when we study the generalization of supplemented module. In addition we obtained the third generalization of this module by use many concepts as injective module; semiperfect module and $D_{i}$-modules. Also we found if $M$ is a $Rad$-$D_{11}$-module with $D_{2}$ property gives $C$-$Rad$-$D_{11}$-module. Moreover; if $M$ is a $C_{3}$-module having the $SIP$ and $C_{3}$ properties lead to $M$ is a $C$-$Rad$-$D_{11}$-module.

\section{Acknowledgement}
The authors would like to acknowledge the financial support received from Universiti Kebangsaan Malaysia under the research Grant UKM-DLP-2013-020. The authors also wish to gratefully acknowledge all those who have generously given of their time to referee our paper.\\

REFERENCES\\

  [1] N. Orhan, D.T. Tutuncu and R. Tribak, On hollow-lifting modules, Taiwanese J. Math., 11, no. 2, (2007) 545-568.\\

  [2] Y. Talebi and A. Mahmoudi, On $Rad$-$\bigoplus$-Supplemented Modules, Thai Journal of Mathematics, 9, no. 2, (2011) 373-381.\\

  [3] A. Harmanci, D. Keskin and P.F. Smith, On $\bigoplus$-supplemented modules, Acta Math. Hungar., 83, (1991) 161-169.\\

  [4] N. Orhan and D.T. Tutuncu, Generalization of Weak Lifting Modules, Soochow. J of Math., 32, no. 1, (2006) 71-76.\\

  [5] M. Alkan and A. Harmanci, On Summand Sum and Summand Intersection Property of Modules, Turk J Math., 26, (2002) 131- 147.\\

  [6] B. Talaee, Generalization of $D_{11}$ and $D^{+}$$_{11}$ Modules. Tarbiat Moallem University, $20^{th}$ Seminar on Algebra, Ordibehesht. 1388, no. 2-3, (2009) 213-216.\\

 [7] N.O. Ertas, (*)-generalized projective module and lifting modules, International Mathematical Forum, 5, no. 2, (2010) 59–68.\\

 [8] Y. Wang and N. Ding, Generalized supplemented modules. Taiwanese. Journal Mathematics, 20, no. 6, (2006) 1589–1601.

  [9] R. Wisbauer, Foundations of Module and Ring Theory, Gordon and Breach, Reading, Philadelphia, (1991).\\

  [10] M. Alkan and A. Harmanci, On Summand Sum and Summand Intersection Property of Modules, Turk. J. Math., 26, (2002) 131- 147.\\

[11] A.C. Ozcan and A. Harmanci, Duo Modules, Glasgow Math. J., 48, (2006) 533-545.\\

  [12] Z. Qing-yi and S. Mei-hua, On closed weak supplemented modules, Journal of Zhejiang University Science, 7, no. 2, (2006) 210-215.

\bibliographystyle{amsplain,latexsym}

\end{document}